\documentclass[10pt,reqno]{amsart}
\usepackage{cmap}
\usepackage[T1]{fontenc}
\usepackage[utf8]{inputenc}
\usepackage[english]{babel}
\usepackage{amsfonts,amssymb,amsthm,amsmath}
\usepackage{url}
\usepackage{enumitem}

\allowdisplaybreaks
\usepackage{mathtools}
\mathtoolsset{showonlyrefs}

\usepackage{secdot}

\usepackage{graphicx}
\usepackage{epstopdf}
\usepackage{a4wide}
\usepackage{xcolor}
\usepackage[colorlinks=true,linktocpage,pdfpagelabels,
bookmarksnumbered,bookmarksopen]{hyperref}
\definecolor{ForestGreen}{rgb}{0.1,0.6,0.05}
\definecolor{EgyptBlue}{rgb}{0.063,0.1,0.6}
\hypersetup{
	colorlinks=true,
	linkcolor=EgyptBlue,         
	citecolor=ForestGreen,
	urlcolor=olive
}

\usepackage[hyperpageref]{backref}

\usepackage{orcidlink}

\newtheorem{theorem}{Theorem}
\newtheorem{proposition}[theorem]{Proposition}
\newtheorem{lemma}[theorem]{Lemma}
\newtheorem{corollary}[theorem]{Corollary}

\theoremstyle{definition}

\newtheorem{remark}[theorem]{Remark}

\numberwithin{equation}{section}
\numberwithin{theorem}{section}
\numberwithin{equation}{section}
\numberwithin{theorem}{section}

\newcommand{\W}{W_0^{1,p}(\Omega)}
\newcommand{\intO}{\int_\Omega}

\newcommand{\doi}[1]{\href{http://dx.doi.org/#1}{DOI:#1}}

\title[Multiplicity of dead core solutions in indefinite elliptic problems]{Multiplicity of dead core solutions in indefinite elliptic problems}

\author[V.~Bobkov]{Vladimir Bobkov}
\author[H.~Ramos Quoirin]{Humberto Ramos Quoirin}

\address[V.~Bobkov]{\newline \indent Institute of Mathematics, Ufa Federal Research Centre, RAS 
\newline \indent
Chernyshevsky str. 112, 450008 Ufa, Russia
\newline \indent
\orcidlink{0000-0002-4425-0218} 0000-0002-4425-0218
}
\email{\href{mailto:bobkov@matem.anrb.ru}{ bobkov@matem.anrb.ru}}

\address[H.~Ramos Quoirin]{\newline \indent CIEM-FaMAF \newline \indent Universidad Nacional de C\'{o}rdoba, \newline\indent 
(5000) C\'{o}rdoba, Argentina
\newline \indent
\orcidlink{0000-0001-6252-2515} 0000-0001-6252-2515
}
\email{\href{mailto:humbertorq@gmail.com}{ humbertorq@gmail.com}}



\begin{document}

\begin{abstract}
We investigate nonnegative solutions of indefinite elliptic problems which enjoy the dead core phenomenon. Our model is the subhomogeneous problem   
$$
-\Delta_p u = (a^+(x) - \mu a^-(x))|u|^{q-2}u,
\quad u \in \W,
$$
where $\Omega$ is a bounded domain in $\mathbb{R}^N$, $1<q<p$, and $\mu > 0$. Thanks to a dead core formation property, we obtain multiple nonnegative dead core solutions of this problem for $\mu$ large enough. This assertion, in combination with a uniqueness feature, yields an exact multiplicity result, which gives a precise description of the nonnegative solutions set of this problem for large values of $\mu$. We also extend the multiplicity result to other classes of problems.
\end{abstract}

\maketitle

\section{Introduction and main results}\label{sec:intro}

Let $\Omega$ be a bounded domain in $\mathbb{R}^N$, $N \geq 1$, $a \in C({\Omega}) \cap L^\infty(\Omega)$ be a sign-changing weight, and $\Delta_p$ be the $p$-Laplace operator with $p>1$. The problem
\begin{equation}\label{eq:P}
        \tag{$\mathcal{P}$}
	\left\{
	\begin{aligned}
		-\Delta_p u &=  a(x)|u|^{q-2}u 
		&&\text{in}\ \Omega, \\
		u &=0 &&\text{on}\ \partial \Omega,
	\end{aligned}
	\right.
\end{equation}
where $q\neq p$ and $1<q<p^*$ (the critical Sobolev exponent), is a prototype of boundary value problem for quasilinear elliptic PDEs with an \textit{indefinite} nonlinearity.
One can identify two classes of problems within \eqref{eq:P} by comparing the exponents $q$ and $p$, namely, the {\it subhomogeneous} problem ($q<p$), and the {\it superhomogeneous} problem ($q>p$). In the semilinear case $p=2$, this terminology reduces to {\it sublinear} and {\it superlinear}, respectively.

Let us recall that weak solutions of \eqref{eq:P} are precisely critical points of the energy functional 
\begin{equation}
I_a(u)
=
\frac{1}{p}\intO |\nabla u|^p \,dx-\frac{1}{q}\intO a\,|u|^q \,dx, \quad u \in \W.
\end{equation}
Classical regularity results for quasilinear elliptic equations (see, e.g., \cite{db} and \cite[Section~3.5]{DKN}) show that any such solution $u$ belongs to $C^{1,\alpha}(\Omega) \cap L^\infty(\Omega)$ for some $\alpha\in(0,1)$. 
Moreover, it follows from \cite[Theorem~1.2]{ciacni-mazja} that 
\begin{equation}\label{eq:W12loc}
|\nabla u|^{p-2} \nabla u \in W^{1,2}_{\text{loc}}(\Omega),
\end{equation}
which implies that $u$ satisfies \eqref{eq:P} a.e.\ in $\Omega$. 
By a {\it nonnegative} solution of \eqref{eq:P} we mean $u\geq 0$ in $\Omega$.
If, in addition, $u>0$ in
$\Omega$, then we call it a \textit{positive} solution. 
We say that $u$ is a {\it dead core} solution if it vanishes in some open subset of $\Omega$. 
Finally, any solution delivering the least value of 
$I_a$ among all solutions is called a \textit{ground state} solution.

The class of indefinite elliptic problems has received considerable attention in the last half-century, mostly in the superhomogeneous case \cite{AT,ALG,BCN,BD,Bonheure1,I,Ou}. On the other hand, we refer to \cite{alama,bandle,BPT,Br} for some seminal works in the sublinear case, and to \cite{BT,KRQU-cv,KRQU-cpaa} (and references therein) for recent contributions in the subhomogeneous case.
As long as nonnegative solutions are concerned, a crucial difference between the subhomogeneous and superhomogeneous cases is related to the following positivity property: if $q>p$, then the strong maximum principle applies (see, e.g., \cite{V}), forcing any nontrivial nonnegative solution of \eqref{eq:P} to be positive.
In the subhomogeneous scenario $q<p$, such a property no longer holds, and nontrivial nonnegative solutions of \eqref{eq:P} may have a {\it dead core}, i.e., an open subset of $\Omega$ where it vanishes. 
Let us note that the above described positivity property holds in both superhomogeneous and subhomogeneous regimes if one deals with the {\it definite} case $a \geq 0$.

Our purpose in this work is to investigate further the structure of the nonnegative solutions set of \eqref{eq:P} in the subhomogeneous case, focusing on the influence of the negative part of the weight $a$ on such a structure. 
We shall then consider a family of problems \eqref{eq:P} in the form
\begin{equation}\label{eq:Pmu}
    \tag{$\mathcal{P}_\mu$}
	\left\{
	\begin{aligned}
		-\Delta_p u &=  a_\mu(x)|u|^{q-2}u 
		&&\text{in}\ \Omega, \\
		u &=0 &&\text{on}\ \partial \Omega,
	\end{aligned}
	\right.
\end{equation}
where $1<q<p$, $a_\mu :=a^+ - \mu a^-$, $a^{\pm}:=\max\{\pm a,0\}$, and $\mu > 0$ is a parameter.

Denote 
\begin{align}
	&\Omega_a^+ = \{x \in \Omega:~ a(x)>0\},\\
	&\Omega_a^{+,0} = \mathrm{Int}(\{x \in \Omega:~ a(x) \geq 0\}),\\
	&\Omega_a^- = \{x \in \Omega:~ a(x)<0\}. 
\end{align}
We impose the following additional assumptions on the sign-changing weight $a \in C({\Omega}) \cap L^\infty(\Omega)$, see Figure~\ref{fig1}:
\begin{enumerate}[label={$\mathbf{(A_{\mathrm{\arabic*}})}$}]
	\item\label{assumptionA1}
    $\Omega_a^{+,0}$ contains precisely $n \geq 1$ connected components $\omega_i$ such that $\omega_i \cap \Omega_a^+ \neq \emptyset$, i.e., for any $i = 1,\dots,n$ there exists $x_i \in \omega_i$ such that $a(x_i)>0$.

	\item\label{assumptionA2} 
	each $\overline{\omega_i}$ is
     surrounded by $\Omega_a^-$, i.e.,  
     for any $i=1,\dots,n$ there exists $\delta_0>0$ such that for any $\delta \in (0,\delta_0)$ and any $x \in \mathbb{R}^N$ satisfying $\text{dist}(x,\overline{\omega_i}) = \delta$ we have $x \in \Omega$ and $a(x)<0$.
\end{enumerate}
Observe that any connected component of $\Omega_a^+$ is a subset of some $\omega_{i}$. 
The assumption \ref{assumptionA2} implies that $\overline{\omega_i} \subset \Omega$ and  $\overline{\omega_i} \cap \overline{\omega_j} = \emptyset$ for any $i \neq j$. 
Note also that $\Omega_a^{+,0}$ may have a connected component $\omega$ where $a \equiv 0$, which occurs, e.g., when $\omega$ is surrounded by $\Omega_a^-$. 

\begin{figure}[!ht]
	\begin{center}
		\includegraphics[width=0.7\linewidth]{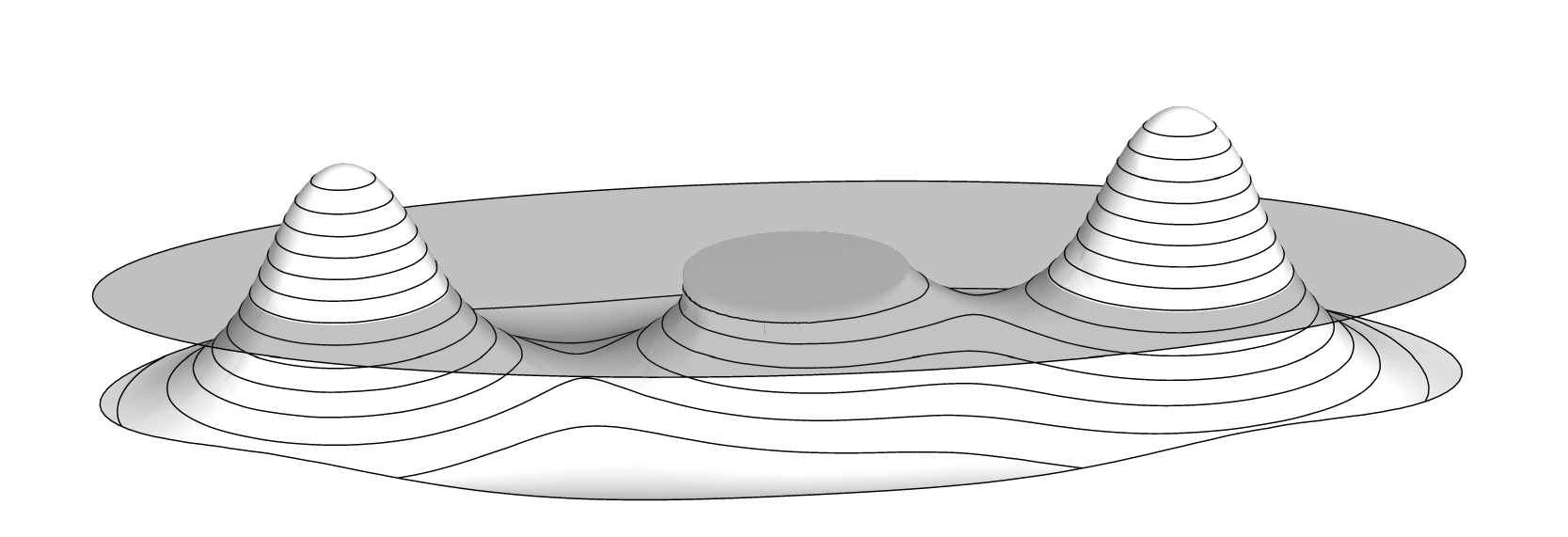}
	\end{center}
	\caption{An example of the weight $a$ satisfying the assumptions \ref{assumptionA1} and \ref{assumptionA2} with $n=2$ and for which $\Omega_a^{+,0}$ has three connected components. The grey hyperplane is the zero level.}
	\label{fig1}
\end{figure}

It is known that for any $\mu>0$, \eqref{eq:Pmu} possesses a unique nonnegative ground state solution $U_\mu$, which is the only nonnegative solution of \eqref{eq:Pmu} being positive in $\Omega_a^+$, see, e.g., \cite[Theorem~1.1]{KRQU-cpaa}. Let us denote by $\mathcal{S}_\mu$ the set of nontrivial nonnegative solutions of \eqref{eq:Pmu}, and by $I_\mu$ the energy functional associated to \eqref{eq:Pmu}, i.e., $I_\mu =I_{a_\mu}$.
In particular, $U_\mu \in \mathcal{S}_\mu$ and $U_\mu$ minimizes $I_\mu$ among all the solutions of \eqref{eq:Pmu}.

Our first main result provides an overall description of $\mathcal{S}_\mu$:

\begin{theorem}[General structure of $\mathcal{S}_\mu$]\label{thm:main0}
Let \ref{assumptionA1} be satisfied. 
Then there exists $\mu_0=\mu_0(p,q,\Omega, a)>0$ such that $\mathcal{S}_\mu=\{U_\mu\}$ for $0<\mu<\mu_0$,  and $\mathcal{S}_\mu$ has at least two elements for $\mu>\mu_0$. Moreover, $\mathcal{S}_\mu$ has at most $2^n-1$ elements for $\mu>0$.
\end{theorem}

Theorem \ref{thm:main0} shows that the singleton feature $\mathcal{S}_\mu=\{U_\mu\}$, which is well known for $\mu=0$ (see \cite{DS}), extends to positive values of $\mu$ up to the threshold $\mu_0$. 
In addition, we note that for small values of $\mu$, $U_\mu$ inherits the positivity in $\Omega$ from $U_0$, at least under a sufficient smoothness condition on $\Omega$, cf.\ \cite[Proposition~4.6]{KRQU-cv}.

In what follows, for an arbitrary subset $\Omega' \subset \Omega$ we denote
\begin{equation}\label{eq:K}
K(\Omega'):= \{u \in \W:~ u=0 \text{ a.e.\ in } \Omega'\},
\end{equation}
and $\chi_{\Omega'}$ shall stand for the characteristic function of $\Omega'$.

Our second main result gives an exact description of $\mathcal{S}_\mu$ for large values of $\mu$. 
More precisely, it shows that the ground state solution $U_\mu$ spans $\mathcal{S}_\mu$ in the following sense:
\begin{theorem}[Precise description of $\mathcal{S}_\mu$ for large $\mu$]\label{thm:main1}
Let \ref{assumptionA1} and \ref{assumptionA2} be satisfied. 
Then there exists $\mu_{\infty}=\mu_{\infty}(p,q,\Omega,a) \geq \mu_0$ such that
$$
\mathcal{S}_\mu=\left\{\sum_{i\in \mathcal{J}}U_{\mu,i}:~ \mathcal{J} \subset \{1,\dots,n\}\right\}
$$
for $\mu>\mu_{\infty}$.
Here $U_{\mu,i}:=U_\mu \, \chi_{\Omega_i}$ are ``bumps'' of $U_\mu$ and $U_\mu=\sum_{i=1}^n U_{\mu,i}$, where each $\Omega_i$ is a subdomain of $\Omega$ satisfying $\overline{\omega_i}\subset \Omega_i$, and $\Omega_i \cap \Omega_j = \emptyset$ for $i \neq j$. 
Moreover, the following properties hold for $i=1,\dots,n$:
\begin{enumerate}[label={\rm(\roman*)}]
    \item $U_{\mu,i} \in \W$, $\overline{\omega_i} \subset \mathrm{supp}\, U_{\mu,i} \subset \overline{\Omega_i}$, $U_{\mu,i}$ solves \eqref{eq:Pmu} and minimizes $I_\mu$ over $K(\Omega \setminus \Omega_i)$. 

    \item 
    The map $\mu \mapsto \mathrm{supp}\, U_{\mu,i}$ is nonincreasing and converges, in the sense of Hausdorff, to  $\overline{\omega_i}$ as $\mu \to +\infty$. 
    
    \item $U_{\mu,i} \to U_{\infty}\chi_{\omega_i}$ in 
$\W$ as $\mu \to +\infty$, where $U_\infty$ is the unique nonnegative minimizer of $I_0$ over $K(\Omega_a^-)$. 
\end{enumerate}
In particular, \eqref{eq:Pmu} possesses exactly $2^n-1$ nontrivial nonnegative solutions (all of which are dead core solutions) for $\mu>\mu_{\infty}$.
\end{theorem}

\begin{remark}
It is clear that if $u$ solves \eqref{eq:Pmu}, then so does $-u$. 
Thus, Theorem \ref{thm:main1} also provides us with sign-changing solutions of \eqref{eq:Pmu} for $\mu>\mu_\infty$. 
Namely, any sum of the form $\sum_{i\in \mathcal{J}_1} U_{\mu,i}-\sum_{i\in \mathcal{J}_2} U_{\mu,i}$ with $\mathcal{J}_1, \mathcal{J}_2 \subset \{1,\dots,n\}$ and $\mathcal{J}_1 \cap \mathcal{J}_2 =\emptyset$
is a sign-changing solution of \eqref{eq:Pmu}. 
We also deduce that for $\mu>\mu_\infty$ the ground state level of $I_\mu$ is achieved both by the nonnegative solution $U_\mu$ and by sign-changing solutions obtained by switching the sign of any proper subset of bumps $\{U_{\mu,i}\}$ in the decomposition of $U_\mu$. Therefore, the ground state level and the nodal ground state level of $I_\mu$ are equal for $\mu>\mu_\infty$.    
\end{remark}

The existence of $2^n-1$ nontrivial nonnegative solutions of the problem \eqref{eq:Pmu} in Theorem~\ref{thm:main1}
is based on the following dead core property of \eqref{eq:Pmu}: given a set $\Omega'$ compactly contained in $\Omega_a^-$, \textit{any} nonnegative solution of \eqref{eq:Pmu} vanishes in $\Omega'$ for $\mu$ large enough. 
By noting that the ground state solution $U_\mu$
is positive in every $\omega_i$, and $\overline{\omega_i} \cap \overline{\omega_j} = \emptyset$ for any $i \neq j$ by the assumption \ref{assumptionA2}, we derive that for sufficiently large $\mu$, $U_\mu$ has an $n$-bump structure, that is,
$U_\mu=\sum_{i=1}^n U_{\mu,i}$, where $\{U_{\mu,i}\}$ is a family of nonnegative bumps with disjoints supports, i.e.,  
$$
U_{\mu,i} \in \W, \quad U_{\mu,i} \geq 0, \quad \text{and} \quad \overline{\omega_i} \subset \mathrm{supp}\, U_{\mu,i} \subset \overline{\Omega_i} \text{ with } \overline{\Omega_i} \cap \overline{\Omega_j} = \emptyset \text{ for any } i \neq j.
$$
We show that any bump $U_{\mu,i}$ solves \eqref{eq:Pmu}, and, combining all the bumps, we obtain
$2^n-1$ nontrivial nonnegative solutions for $\mu$ large enough. 

Theorems~\ref{thm:main0} and \ref{thm:main1} generalize the findings from \cite[Section~2]{bandle} to the quasilinear case $p \neq 2$. 
However, our arguments are somewhat different and rather variational in nature. 
In contrast to \cite{bandle}, we do not impose additional smoothness assumptions on $\omega_i$, cf.\ \cite[Assumption (H)]{bandle}. 
On the other hand, the nonlinear dependence on $u$ and the boundary condition considered in \cite{bandle} are more general.

In the superlinear case $2=p<q<2^*$, in which the strong maximum principle holds, a counterpart of Theorem~\ref{thm:main1} was obtained in \cite{Bonheure1} under the essentially same assumptions on the weight $a$. More precisely, \cite[Theorem~1.1]{Bonheure1} states the existence of \textit{at least} $2^n-1$ positive solutions of \eqref{eq:Pmu} for any sufficiently large $\mu$, with a related convergence property of their supports, and we refer to \cite{AZ,BDG,FZ,GG} for developments and investigation of related issues. 
At the same time, the only \textit{exact} multiplicity result in the superlinear regime has been proved very recently in the ODE setting by \cite{FT}. 
To the best of our knowledge, no exact multiplicity results are available for $p \neq 2$, both in the superhomogeneous and subhomogeneous cases. 
Although the conclusion of Theorem~\ref{thm:main1} is close to the result stated in \cite{Bonheure1}, the subhomogeneous case $q<p$ is considerably different in nature and the arguments are essentially different.

\smallskip
Let us remark that the dead core property described above holds, at least partially, for a more general class of problems. More precisely, given a ball $B$ and a sufficiently large constant $A>0$, any family of nonnegative subsolutions of $-\Delta_p u =-Au^{q-1}$ in $B$ which, in addition, is uniformly bounded on $\partial B$, enjoys the dead core property, see Section~\ref{sec:proof}.
As a consequence, we show that the multiplicity result of Theorem~\ref{thm:main1} has the following extension:

\begin{theorem}\label{t2}
Let \ref{assumptionA1} and  \ref{assumptionA2} be satisfied, and $1<q<p\leq r$. Then there exists  $\delta>0$ such that whenever $\|a^+\|_\infty \in (0,\delta)$ the following result holds: there exists $\mu_{\infty}=\mu_{\infty}(p,q,r,\Omega, a)>0$ such that the problem
\begin{equation}\label{eq:linear}
-\Delta_p u = a_\mu(x)(|u|^{q-2}u+|u|^{r-2}u), \quad u \in \W, 
\end{equation}
has at least $2^n-1$ nontrivial nonnegative solutions (all of which are dead core solutions) for $\mu>\mu_{\infty}$. Moreover, as $\mu \to +\infty$, the support of each such solution converges, in the sense of Hausdorff, to $\bigcup_{i\in \mathcal{J}} \overline{\omega_i}$ for some $\mathcal{J} \subset \{1,\dots,n\}$. 
\end{theorem}
 
 We observe that the $2^n-1$ solutions of the problem \eqref{eq:linear} given in Theorem~\ref{t2} can be generated from a solution obtained through a nonvariational construction, which is especially relevant in the supercritical case $r\geq p^*$ (when $p<N$). In this case, we rely on the sub- and supersolutions method to obtain a solution which is positive in $\bigcup_{i} \omega_i$ and zero in $\Omega \setminus \bigcup_i {\Omega_i}$, see Remark~\ref{rf} below.

The rest of this work is organized as follows. 
In Section~\ref{sec:uniqueness}, we provide auxiliary uniqueness results. 
In Section~\ref{sec:properties}, we collect a few important properties of ground state solutions of \eqref{eq:Pmu}, and we prove Theorem \ref{thm:main0}.
Section~\ref{sec:proof} is devoted to the dead core phenomenon and the proof of Theorem~\ref{thm:main1}. 
Finally, in Section~\ref{sec:extensions}, we discuss extensions of Theorem~\ref{thm:main1} to other nonlinearities and prove Theorem \ref{t2}.

\section{Uniqueness results}\label{sec:uniqueness}
We start by showing that the problem \eqref{eq:Pmu} has at most one nonnegative solution $u$ if, among the connected components $\omega_{i}$ from the assumption \ref{assumptionA1}, we prescribe those where $u$ vanishes and those where $u$ is positive.  
Recall that for an arbitrary subset $\Omega' \subset \Omega$ we denote
\begin{equation}
K(\Omega'):= \{u \in \W:~ u=0 \text{ a.e.\ in } \Omega'\}.
\end{equation}
It is not difficult to see that $K(\Omega')$ is nonempty and weakly closed in $\W$. 

The following two results are essentially contained in \cite[Propositions~2.3 and 2.7]{KRQU-cpaa}, which we refer to for more details on the proofs.
\begin{proposition}\label{p1}
Let $b \in L^{\infty}(\Omega)$ and $\Omega' \subset \Omega$. Then the functional 
$$
I_b(u):=\frac{1}{p}\intO |\nabla u|^p\,dx 
-
\frac{1}{q}\intO b \,|u|^q\,dx, 
\quad u \in \W,
$$
has exactly one nonnegative minimizer $v$ over $K(\Omega')$, and $v \in L^\infty(\Omega)$. 
In particular, we have $I_b'(v)\phi=0$ for any $\phi \in K(\Omega')$.
\end{proposition}
\begin{proof}
Since $I_b$ is even, coercive, and weakly lower semicontinuous, and the set $K(\Omega')$ is weakly closed, a classical minimization argument shows the existence of a nonnegative $u \in K(\Omega')$ such that $I_b(u)=\min_{K(\Omega')} I_b$. 
Assume now that $0\leq v \in K(\Omega')$ satisfies $I_b(v)= I_b(u)$ and set $w=\left(\frac{1}{2}(u^q+v^q)\right)^{\frac{1}{q}}$. 
Evidently, we have $0 \leq w \in K(\Omega')$. 
The generalized hidden convexity principle (see, e.g., \cite[Theorem~2.9 and Remark~2.10]{brasco}) gives
$$
I_b(w)\leq \frac{1}{2}\left(\frac{1}{p}\intO (|\nabla u|^p+|\nabla v|^p)\,dx -\frac{1}{q}\intO b\,(u^q+v^q)\,dx\right)=\frac{1}{2}(I_b(u)+I_b(v))=\min_{K(\Omega')} I_b,
$$
i.e., $I_b(w)=\min_{K(\Omega')} I_b$. 
Since $u,v \in \W$, the characterization of equality cases in \cite[Theorem~2.9]{brasco} implies that $u \equiv v$. 

Observe that the standard bootstrap argument (see, e.g., \cite[Theorem~3.5]{DKN}) is based on the choice of a truncated function $v_M = \max\{u,M\}$ with $M>0$ as a test function in $I_b'(v)\phi=0$. Since $v_M \in K(\Omega')$, this choice is admissible in the present settings, leading to the boundedness of $v$. 
\end{proof}

\begin{remark}\label{r1}
Clearly, the minimizer $v$ of $I_b$ over $K(\Omega')$ is nonzero if and only if $b^+ \not \equiv 0$ in $\Omega \setminus \Omega'$. 
It should be noted that, in general, $v$ is not a critical point of $I_b$ over $\W$. 
On the other hand, if $\Omega'$ is a closed set and $\Omega \setminus \Omega'$ is sufficiently smooth (more precisely, $\Omega \setminus \Omega'$ is a $p$-stable set), then $K(\Omega')$ can be identified with $W_0^{1,p}(\Omega \setminus \Omega')$, see, e.g., \cite[Section 3.4.2]{hp}. 
In this case, $v$ is a nonnegative ground state solution of the problem 
$$
-\Delta_p u=b(x) |u|^{q-2}u, 
\quad u \in W_0^{1,p}(\Omega \setminus \Omega').
$$
\end{remark}

Recall that the functional $I_\mu$ is given by
$$
I_\mu(u)=\frac{1}{p}\intO |\nabla u|^p\,dx 
-
\frac{1}{q}\intO (a^+ - \mu a^-)|u|^q\,dx, 
\quad u \in \W.
$$
\begin{proposition}\label{p2}
	Let \ref{assumptionA1} be satisfied and 
	$\mu>0$. 
Assume that the functional $I_\mu$ has a critical point $u \geq 0$ such that 
\begin{equation}\label{eq:Ju}
u>0 ~\text{in}~ \bigcup_{i \in \mathcal{J}} \omega_i
\quad \text{and} \quad 
u \equiv 0 ~\text{in}~ \bigcup_{i \not \in \mathcal{J}} \omega_i
\end{equation}
for some $\mathcal{J} \subset \{1, \dots, n\}$. 
Then $u$ is the unique nonnegative minimizer 
of $I_\mu$ over the set $K(\bigcup_{i \not \in  \mathcal{J}} \omega_i)$. 
In particular, for any $\mathcal{J} \subset \{1, \dots, n\}$ the problem \eqref{eq:Pmu} has at most one solution $u$ satisfying \eqref{eq:Ju}. 
\end{proposition}

\begin{proof}
We know from Proposition~\ref{p1} that $I_\mu$ has a unique nonnegative minimizer $v \in \W \cap L^\infty(\Omega)$ over the set $K(\bigcup_{i \not \in  \mathcal{J}} \omega_i)$.
Let us prove that $v \equiv u$. 
For any $\varepsilon>0$ we consider the function $\phi=\frac{v^q}{(u+\varepsilon)^{q-1}}$. 
One can argue as in the proof of \cite[Lemma~B.1]{BTant} to justify that $\phi \in \W$. 
Consequently, using $\phi$ as a test function in $I_\mu'(u)\phi=0$, we obtain
$$
\intO a_\mu v^q \left(\frac{u}{u+\varepsilon}\right)^{q-1} \,dx=\intO |\nabla u|^{p-2} \nabla u \nabla \left(\frac{v^q}{(u+\varepsilon)^{q-1}}\right) dx. 
$$
The generalized Picone identity (see, e.g., \cite[Proposition~2.9]{brasco2}) gives
$$
|\nabla u|^{p-2} \nabla u \nabla \left(\frac{v^q}{(u+\varepsilon)^{q-1}}\right)\leq |\nabla u|^{p-q} |\nabla v|^q 
\quad \text{in}~\Omega, 
$$
and hence the H\"older inequality yields 
$$
\intO a_\mu v^q\left(\frac{u}{u+\varepsilon}\right)^{q-1} \,dx
\leq 
\intO |\nabla u|^{p-q} |\nabla v|^q \,dx
\leq 
\| \nabla u\|_p^{p-q} \|\nabla v\|_p^q.
$$ 
Letting $\varepsilon \to 0$, we see that
\begin{equation}\label{eq:pr2:-1}
\int_{[u>0]} a_\mu v^q \,dx
= 
\lim_{\varepsilon \to 0} 
\intO a_\mu v^q\left(\frac{u}{u+\varepsilon}\right)^{q-1} \,dx
\leq \| \nabla u\|_p^{p-q} \|\nabla v\|_p^q,
\end{equation}
where $[u>0]:=\{x \in \Omega: u(x)>0\}$. 
Denoting also $[u=0]:=\{x \in \Omega: u(x)=0\}$ and recalling that $u, v \in K(\bigcup_{i \not \in  \mathcal{J}} \omega_i)$, we get 
$$
\int_{[u=0]} a_\mu v^q \,dx
\leq 
\int_{[u=0]} a^+ v^q  \,dx
=
\int_{[u=0] \cap \bigcup_{i \not \in  \mathcal{J}} \omega_i} a^+ v^q \,dx
+ 
\int_{[u=0] \cap \Omega_a^-} a^+ v^q \,dx
= 0.
$$
Therefore, we obtain from \eqref{eq:pr2:-1} that 
\begin{equation}\label{eq:pr2:0}
\intO a_\mu v^q  \,dx
\leq 
\int_{[u>0]} a_\mu v^q \,dx
\leq 
\| \nabla u\|_p^{p-q} \|\nabla v\|_p^q.
\end{equation}
Finally, since $I_\mu'(u)u=I_\mu'(v)v=0$, we have 
\begin{equation}\label{eq:pr2:1}
\|\nabla u\|_p^p
=
\intO a_\mu u^q \,dx
\quad \text{and} \quad 
\|\nabla v\|_p^p=\intO a_\mu v^q \,dx.
\end{equation}
In particular, \eqref{eq:pr2:0} and \eqref{eq:pr2:1} imply
\begin{equation}\label{eq:pr2:2}
\|\nabla v\|_p^p \leq \| \nabla u\|_p^{p-q} \|\nabla v\|_p^q, 
\quad \text{i.e.,} \quad \|\nabla v\|_p\leq \|\nabla u\|_p.
\end{equation}
It then follows from \eqref{eq:pr2:1} and \eqref{eq:pr2:2} that 
\begin{equation}\label{eq:pr2:3}
I_\mu(u)=-\frac{p-q}{pq}\|\nabla u\|^p \leq -\frac{p-q}{pq}\|\nabla v\|^p=I_\mu(v).
\end{equation}
Thus, $u$ is also a minimizer 
of $I_\mu$ over the set $K(\bigcup_{i \not \in  \mathcal{J}} \omega_i)$, and hence 
Proposition~\ref{p1} yields $u \equiv v$.
\end{proof}

Let us denote by
$V_i$ the unique nonnegative ground state solution of the problem $$-\Delta_p u=a(x) |u|^{q-2}u, \quad u \in W_0^{1,p}(\omega_i),$$ for $i=1,\dots,n$, where the sets $\omega_{i}$ are given in the assumption~\ref{assumptionA1}; see, e.g., \cite{DS}.
It is not hard to see that each $V_i$ satisfies
$$
I_\mu(V_i)=I_0(V_i)=\min_{W_0^{1,p}(\omega_i)} I_0 < 0.
$$
Since $a\geq 0$ and $a \not\equiv 0$ in $\omega_i$, we have $V_i>0$ in $\omega_i$ by the strong maximum principle \cite{V}.
Moreover, we assume $V_i$ to be extended by zero to the whole $\Omega$, so that $V_i \in \W$. 

\begin{lemma}\label{pg1}
	Let \ref{assumptionA1} be satisfied and $\mu>0$. Then any nontrivial nonnegative solution $u$ of \eqref{eq:Pmu} satisfies 
\begin{equation}\label{eq:pg1}
	I_0(u)\leq I_\mu(u)\leq \max\{I_0(V_i):~ i=1,\dots,n\} < 0.
\end{equation}
\end{lemma}

\begin{proof}
We argue in much the same way as in the proof of Proposition~\ref{p2}. Given a nontrivial nonnegative solution $u$ of \eqref{eq:Pmu}, we know from the strong maximum principle that $u>0$ in $\omega_i$ for some $i \in \{1,\dots,n\}$. Taking now $\phi=\frac{V_i^q}{(u+\varepsilon)^{q-1}}$ as a test function in $I_\mu'(u)\phi=0$, we find that 
$$
\int_{[u>0]} a_\mu V_i^q \,dx
\leq 
\| \nabla u\|_p^{p-q} \|\nabla V_i\|_p^q,
$$ 
which yields
$$
\intO a_\mu V_i^q \,dx
=
\int_{\omega_i}a_\mu V_i^q \,dx
=
\int_{[u>0]} a_\mu V_i^q \,dx
\leq 
\| \nabla u\|_p^{p-q} \|\nabla V_i\|_p^q.
$$
Further arguing as in Proposition~\ref{p2} (see \eqref{eq:pr2:3}), we get $I_\mu(u) \leq I_\mu(V_i)=I_0(V_i)$. 
Finally, recalling that $I_0(V_i)<0$, we arrive at \eqref{eq:pg1}.
\end{proof}

Recall that $\mathcal{S}_\mu$ stands for the set of nontrivial nonnegative solutions of the problem \eqref{eq:Pmu}.

\begin{proposition}\label{pg2}\strut
	Let \ref{assumptionA1} be satisfied and $\mu>0$.
	Then the following assertions hold:
\begin{enumerate}[label={\rm(\arabic*)}]
    \item \label{pg21} $\mathcal{S}_\mu$ has at most $2^n-1$ elements. 
    \item \label{pg22} The cardinality of $\mathcal{S}_\mu$ is nondecreasing with respect to $\mu$.
    \item \label{pg23} The set $\bigcup_{\mu>0} \mathcal{S}_\mu$ is bounded in $\W \cap L^\infty(\Omega)$ and separated from zero in $\W$.
\end{enumerate}
\end{proposition}

\begin{proof}\strut
\ref{pg21} 
The strong maximum principle implies that for any $u \in \mathcal{S}_\mu$ there exists $\mathcal{J} \subset \{1, \dots, n\}$ such that 
$u>0$ in $\bigcup_{i \in \mathcal{J}} \omega_i$
and 
$u=0$ in $\bigcup_{i \not \in \mathcal{J}} \omega_i$. By Proposition \ref{p2} there is exactly one such $u \in \mathcal{S}_\mu$  for every $\mathcal{J} \subset \{1, \dots, n\}$. Since the set $\{1,\dots,n\}$ has $2^n-1$ nonempty subsets, we deduce the desired conclusion.

\ref{pg22} 
Let $\mu>0$. 
As in the proof of \ref{pg21} above, any $u \in \mathcal{S}_\mu$ is uniquely characterized by a set $\mathcal{J} \subset \{1, \dots, n\}$ such that 
$u>0$ in $\bigcup_{i \in \mathcal{J}} \omega_i$
and 
$u=0$ in $\bigcup_{i \not \in \mathcal{J}} \omega_i$. 
Therefore, in order to prove the monotonicity of the cardinality of $\mathcal{S}_\mu$ with respect to $\mu$, it is sufficient to show that for any $\mu'>\mu$, $u$ gives rise to a solution $v \in \mathcal{S}_{\mu'}$ satisfying
$v>0$ in $\bigcup_{i \in \mathcal{J}} \omega_i$
and 
$v=0$ in $\bigcup_{i \not \in \mathcal{J}} \omega_i$.

Fix any $\mu'>\mu$. It is clear that $u$ is a supersolution of \hyperref[eq:Pmu]{$(\mathcal{P}_{\mu'})$}. 
To obtain a subsolution, we consider a function $\eta = c \sum_{i \in \mathcal{J}} \phi_{1,B_i}$, where $c>0$ and $\phi_{1,B_i}$ is the positive first eigenfunction of $-\Delta_p$ in $W_0^{1,p}(B_i)$ with $\|\phi_{1,B_i}\|_\infty =1$, and a ball $B_i$ satisfies $\overline{B_i}\subset \omega_i \cap \Omega_a^+$, i.e., $a>0$ in $\overline{B_i}$.
Extending $\eta$ by zero to $\Omega$ and choosing $c$ small enough, it can be shown (see, e.g., the proof of \cite[Theorem~2.1]{DGU3}) that $\eta$ becomes a subsolution of \hyperref[eq:Pmu]{$(\mathcal{P}_{\mu'})$}. 
Furthermore, we can choose $c$ smaller to get $\eta \leq u$.  
Then the sub- and supersolutions method gives the existence of a solution $v$ of \hyperref[eq:Pmu]{$(\mathcal{P}_{\mu'})$} lying between $\eta$ and $u$. 
In particular, $v$ is nontrivial and nonnegative, 
$v > 0$ in $\omega_i$ for $i \in \mathcal{J}$  by the strong maximum principle, and $v \equiv 0$ in $\omega_i$ for $i \not\in \mathcal{J}$. 
This finishes the proof. 

\ref{pg23} 
For any $\mu>0$ and $u \in \mathcal{S}_\mu$, the Rellich-Kondrachov theorem implies that
$$
\|\nabla u\|_p^p
=
\intO \left(a^+ - \mu a^-\right) u^q \,dx
\leq
\intO a^+ u^q \,dx
\leq C \|a^+\|_\infty \|\nabla u\|_p^q, 
$$
where $C>0$ does not depend on $\mu$ and $u$. 
This yields the boundedness of $\bigcup_{\mu>0} \mathcal{S}_\mu$ in $\W$. The bootstrap argument (see, e.g., \cite[Theorem~3.5]{DKN}) gives the boundedness in $L^\infty(\Omega)$.
The separation of $\bigcup_{\mu>0} \mathcal{S}_\mu$ from zero in $\W$ follows from Lemma~\ref{pg1} 
and the continuity of $I_\mu$ in $\W$. 
\end{proof}

\section{Properties of the ground state solution and the proof of Theorem~\ref{thm:main0}}\label{sec:properties}

Recall that $U_\mu$ is the unique nonnegative ground state solution of \eqref{eq:Pmu}. More precisely, we have
\begin{equation}\label{eq:ground}
U_\mu \in \W, \quad U_\mu \geq 0, \quad \text{and} \quad I_\mu(U_\mu)=m(\mu):=\min_{\W} I_\mu,
\end{equation}
and $U_\mu$ is the only solution of \eqref{eq:Pmu} that is positive in $\Omega_a^+$, see, e.g., \cite[Theorem 1.1]{KRQU-cpaa}.
We shall extend this positivity result to the set $\bigcup_{i} \omega_i$ (see the assumption~\ref{assumptionA1}), and also analyse the behaviour of the maps $\mu \to U_\mu$ and $\mu \to m(\mu)$. 

Let us denote by $U_\infty$ the unique nonnegative minimizer of $I_0$ over the set $K(\Omega_a^-) = \{u \in \W:\, u=0 \text{ a.e.\ in } \Omega_a^-\}$ given by Proposition~\ref{p1}. 
\begin{proposition}\label{prop:prop}\strut
Let \ref{assumptionA1} be satisfied and $\mu>0$.
Then the nonnegative ground state solution $U_\mu$ of \eqref{eq:Pmu} has the following properties:
\begin{enumerate}[label={\rm(\arabic*)}]
    \item\label{prop:prop:0} $U_\mu>0$ in $\omega_i$ for any $i=1,\dots,n$. 
        
    \item\label{prop:prop:1} The map $\mu \mapsto U_\mu$ is continuous in $\W$ and nonincreasing pointwisely in $\Omega$, i.e., $U_\mu \to U_{\mu'}$ in $\W$ if $\mu \to \mu'>0$, and $U_\mu \geq U_{\mu'}$ for any $0<\mu<\mu'$. Moreover, $U_\mu \to U_\infty$ in $\W$ as $\mu \to +\infty$. 
	
	\item\label{prop:prop:2} The map $\mu \mapsto m(\mu)$ is concave, increasing, and differentiable with 
	$$
	m'(\mu)=\frac{1}{q}\intO a^- \, U_\mu^q \,dx.
	$$
\end{enumerate}	
\end{proposition}
\begin{proof}\strut
\ref{prop:prop:0} 
We present an explicit argument for the sake of clarity. 
The strong maximum principle says that for any $i =1,\dots,n$ either $U_\mu > 0$ or $U_\mu \equiv 0$ in $\omega_i$. 
Suppose that $U_\mu \equiv 0$ in some $\omega_i$. 
From \ref{assumptionA1} we know that $\omega_i$ intersects $\Omega_a^+$. Taking any $\xi \in C_0^\infty(\Omega)$ supported in $\Omega_a^+ \cap \omega_i$ and recalling that $q<p$, it is not hard to see that $I_\mu(t\xi) < 0$ for sufficiently small $t>0$. Consequently, we have $I_\mu(U_\mu + t\xi) < I_\mu(U_\mu)$, which contradicts \eqref{eq:ground}. 
Thus, $U_\mu>0$ in each $\omega_i$. 

\ref{prop:prop:1} One can argue as in \cite[Proposition 3.7~(2)]{KRQU-cv} to show that $\mu \mapsto U_\mu$ is continuous in $\W$.  
Let us prove that this map is nonincreasing.
If $\mu<\mu'$, then $U_\mu$  is  clearly a supersolution of \hyperref[eq:Pmu]{$(\mathcal{P}_{\mu'})$}. 
As in the proof of Proposition~\ref{pg2} \ref{pg22},
for $i =1,\dots,n$ let $\phi_{1,B_i}$ be the  positive first eigenfunction of $-\Delta_p$ in $W_0^{1,p}(B_i)$ with $\|\phi_{1,B_i}\|_\infty =1$, where a ball $B_i$ satisfies $\overline{B_i}\subset \omega_i \cap \Omega_a^+$, i.e., $a>0$ in $\overline{B_i}$.
One can check that the function $\eta = c \sum_i \phi_{1,B_i}$ extended by zero to $\Omega$ is a subsolution of \hyperref[eq:Pmu]{$(\mathcal{P}_{\mu'})$} for any sufficiently small $c>0$ (see, e.g., the proof of \cite[Theorem~2.1]{DGU3}). 
Further decreasing $c>0$, we can additionally guarantee that $\eta$ is smaller than $U_\mu$. 
The sub- and supersolutions method then yields the existence of a solution $u$ of \hyperref[eq:Pmu]{$(\mathcal{P}_{\mu'})$} satisfying $\eta \leq u \leq U_\mu$ in $\Omega$. Thus, $u \geq 0$ and $u \not \equiv 0$ in $\omega_i$ for every $i$. By the strong maximum principle we have $u>0$ in $\bigcup_{i} \omega_i$, so that the uniqueness of $U_{\mu'}$ yields $u \equiv U_{\mu'}$ and, consequently, $U_{\mu'} \leq U_{\mu}$. 

Let $\mu_k \to +\infty$ and denote $u_k:=U_{\mu_k}$. 
By Proposition~\ref{pg2} \ref{pg23} we know that $\{u_k\}$ is bounded in $\W$. 
Consequently, up to a subsequence, $u_k \to u_0$ weakly in $\W$ and strongly in $L^q(\Omega)$, for some nonnegative $u_0 \in \W$. 
In addition, from $\mu_k \to +\infty$ and the equality
$$
\|\nabla u_k\|_p^p
=
\intO \left(a^+ - \mu_k a^-\right) u_k^q \,dx
$$ 
it follows that $\intO a^- u_0^q \,dx=0$, and hence $u_0 \in K(\Omega_a^-)$.
In view of the convergence properties of $\{u_k\}$, the inequality $\mu_k \intO a^- u_k^q \,dx \geq 0$, and the minimality of $u_k$, we have 
\begin{align}
I_0(u_0)
&=
\frac{1}{p}\|\nabla u_0\|_p^p - \frac{1}{q}\intO a^+ u_0^q \,dx
\\
&\leq
\liminf_{k \to +\infty}
\left(
\frac{1}{p}\|\nabla u_k\|_p^p - \frac{1}{q}\intO a^+ u_k^q \,dx + 
\frac{\mu_k}{q} \intO a^- u_k^q \,dx
\right) 
\\
&= 
\liminf_{k \to +\infty} I_{\mu_k}(u_k) 
\leq \liminf_{k \to +\infty} I_{\mu_k}(u)
=
I_0(u)
\quad \text{for any}~ u \in K(\Omega_a^-),
\end{align}
i.e., $u_0$ minimizes $I_0$ over $K(\Omega_a^-)$. 
Thanks to Proposition~\ref{p1}, this minimizer is unique, and hence $u_0 \equiv U_\infty$. 
Since $I_0'(u_0)u_0=0$ and $I_{\mu_k}'(u_k)u_k=0$, we use the convergence properties of $\{u_k\}$ and the superadditivity of the limit inferior to obtain 
\begin{align*}
\|\nabla u_0\|_p^p
=
\intO a^+u_0^q \,dx
&=
\lim_{k \to +\infty} \intO a^+u_k^q \,dx
\\
&\geq 
\lim_{k \to +\infty} \intO (a^+-\mu_k a^-)u_k^q \,dx
\geq 
\liminf_{k \to +\infty} \|\nabla u_k\|_p^p \geq \|\nabla u_0\|_p^p,
\end{align*}
that is, $\|\nabla u_k\|_p^p \to \|\nabla u_0\|_p^p$. 
Consequently, the uniform convexity of $\W$ implies that $u_k \to u_0$ strongly in $\W$.

\ref{prop:prop:2} The concavity of $\mu \mapsto m(\mu)$ follows from the fact that $I_\mu$ is affine (and therefore concave) with respect to $\mu$, so $m(\mu)$ is the pointwise minimum of a family of concave functions.
In particular, $\mu \mapsto m(\mu)$ is continuous. Finally,
$$
m(\mu)=I_\mu(U_\mu)=I_{\mu'}(U_\mu)+\frac{\mu-\mu'}{q}
\intO a^-U_\mu^q \,dx
\geq
m(\mu')+\frac{\mu-\mu'}{q}\intO a^-U_\mu^q \,dx,
$$
and, in a similar way, 
$$
m(\mu')\geq m(\mu)+\frac{\mu'-\mu}{q}\intO a^-U_{\mu'}^q \,dx.
$$
Thus, we have
$$
\frac{1}{q}\intO a^-U_{\mu'}^q \,dx
\geq \frac{m(\mu)-m(\mu')}{\mu-\mu'}
\geq \frac{1}{q}\intO a^-U_{\mu}^q \,dx
$$ 
provided $\mu>\mu'$, with reverse inequalities holding for $\mu<\mu'$. Letting $\mu' \to \mu$ and recalling the convergence of $\{U_{\mu}\}$ established in \ref{prop:prop:1} above, we obtain the desired conclusion.
\end{proof}

\begin{remark}\label{rem:Uinfty}
Let, in addition, \ref{assumptionA2} be satisfied. 
Since $a \leq 0$ in $\Omega \setminus \bigcup_{i} \omega_i$ and each $\omega_{i}$ is surrounded by $\Omega_a^-$, it is not hard to see that $U_\infty$ minimizes $I_0$ over the set $K(\Omega \setminus \bigcup_{i} \omega_i)$. 
Indeed, any nonzero function supported in $\Omega \setminus \bigcup_{i} \omega_i$ would increase $I_0$, so that $U_\infty \equiv 0$ therein. 
In particular, if $\Omega_a^{+,0}$ has a connected component $\omega$ such that $a\equiv 0$ in $\omega$, then $U_\infty=0$ in $\omega$. 
As observed in Remark~\ref{r1}, if $\bigcup_{i} \omega_i$ is sufficiently smooth, then $U_\infty$ is the unique positive solution of the problem $-\Delta_p u=a^+ |u|^{q-2}u$, $u \in W_0^{1,p}(\bigcup_{i} \omega_i)$.
\end{remark}

\begin{proof}[Proof of Theorem~\ref{thm:main0}]
First, we show that for any sufficiently small $\mu>0$, $U_\mu$ is the only nontrivial nonnegative solution of \eqref{eq:Pmu}.
Suppose, by contradiction, that $\mu_k \to 0^-$ and $u_k$ is a nontrivial nonnegative solution of 
\hyperref[eq:Pmu]{$(\mathcal{P}_{\mu_k})$}
and $u_k \not \equiv U_{\mu_k}$ for every $k$. Propositions~\ref{p2} and~\ref{prop:prop} \ref{prop:prop:0} imply that $u_k$ vanishes in $\omega_{i(k)}$ for some $i(k) \in \{1,\dots,n\}$. Since the set $\{\omega_{i}\}$ has finite cardinality, we can assume, without loss of generality, that $\omega_{i(k)}=\omega_1$ for every $k$.
By Proposition~\ref{pg2}~\ref{pg23} we know that $\{u_k\}$ is bounded in $\W$. 
Therefore, we may assume that $u_k \to u$ weakly in $\W$ and a.e.\ in $\Omega$, for some nonnegative $u \in \W$. 
Then it is not hard to see that $u_k \to u$ strongly in $\W$ and $u$ weakly solves $-\Delta_p u=a^+(x) u^{q-1}$. Using Proposition~\ref{pg2}~\ref{pg23}, we also deduce that $u$ is nontrivial, so that by the strong maximum principle $u>0$ in $\Omega$, which contradicts the convergence $u_k \to u$ a.e.\ in $\Omega$ and the equality $u_k \equiv 0$ in $\omega_1$. 
Therefore, we have $\mathcal{S}_\mu = \{U_\mu\}$ for $\mu>0$ small enough. 
We set $\mu_0=\sup\{\mu>0:\,  \mathcal{S}_\mu=\{U_\mu\}\}$. 
Proposition~\ref{pg2}~\ref{pg22} implies that $\mathcal{S}_\mu=\{U_\mu\}$ for every $\mu<\mu_0$, and $\mathcal{S}_\mu$ has at least two elements for any $\mu>\mu_0$.  
Moreover, the cardinality of $\mathcal{S}_\mu$ is bounded by $2^n-1$ by Proposition~\ref{pg2}~\ref{pg21}.
\end{proof}

\section{A dead core property and the proof of Theorem~\ref{thm:main1}}\label{sec:proof}

As mentioned in Section~\ref{sec:intro}, in the proof of Theorem~\ref{thm:main1} we shall rely on the dead core phenomenon of \eqref{eq:Pmu}. 
We state it in a rather general setting, with the intention to cover other nonlinearities than the purely subhomogeneous one.

We deal with the inequality $-\Delta_p u \leq -Au^{q-1}$ in $B$,
where $B=B(x_0,R)$ is the open ball of radius $R$ centred at $x_0 \in \mathbb{R}^N$, and $A>0$ is a constant. The above inequality
is understood in the weak sense, i.e., $u \in W^{1,p}(B)$ satisfies
$$
\intO |\nabla u|^{p-2} \nabla u \nabla \phi \,dx
\leq 
-A\intO u^{q-1} \phi \,dx \quad \text{for any}~ \phi \in W_0^{1,p}(B), \, \phi \geq 0.
$$
Let 
$$
\mathcal{S}_A:=\{u \in W^{1,p}(B):~ u \geq 0, ~ -\Delta_p u \leq -Au^{q-1} \text{ in } B\}
$$
and
$$
\mathcal{S}_{A,M} :=\{u \in \mathcal{S}_A:~ \sup_{\partial B} u \leq M\} \quad \text{for } M>0.
$$
Note that $u \equiv 0 \in \mathcal{S}_{A,M}$ for any $A,M>0$, so that $\mathcal{S}_{A,M}$ is always nonempty.

The next two results are inspired by \cite[Proposition~5.2]{KRQU-cv}, which we refer to for additional details.
\begin{lemma}\label{dc}
For any $M>0$ and $0<r<R$ there exists $A_0>0$ such that $u \equiv 0$ in $B(x_0,r)$ for any $u \in \mathcal{S}_{A,M}$ with $A>A_0$.
\end{lemma}
\begin{proof}
Let $W(x)=K(|x-x_0|^2-r^2)^{\beta}$ be defined in the spherical shell $B \setminus B(x_0,r)$ with $\beta=p/(p-q)$ and a positive constant $K$. We extend $W$ by zero to $B(x_0,r)$ and choose $K$ large enough so that $W\geq M$ on $\partial B$. One can check that $-\Delta_p W \geq -AW^{q-1}$ in $B$ if $A$ is large enough, and hence the weak comparison principle  implies that any $u \in \mathcal{S}_{A,M}$ satisfies $u\leq W$ in $B$, see, e.g., the proof of \cite[Proposition~5.2]{KRQU-cv} for more details.
In particular, $u \equiv 0$ in $B(x_0,r)$.
\end{proof}

\begin{proposition}\label{cor:wcp}
Let \ref{assumptionA1} and \ref{assumptionA2} be satisfied. 
Let $\widetilde{\Omega} \subset \Omega$ be a neighbourhood of $\bigcup_{i} \omega_i$. 
Then there exists $\mu^*>0$ such that any nonnegative solution of \eqref{eq:Pmu} vanishes in $\Omega \setminus \widetilde{\Omega}$ for $\mu>\mu^*$.
\end{proposition}
\begin{proof}
Let $\Omega_1, \Omega_2$ be two open sets such that 
$$
\textstyle \bigcup_{i} \overline{\omega_i} \subset \Omega_1, \quad  \overline{\Omega_1} \subset \Omega_2, \quad  \overline{\Omega_2} \subset \Omega.
$$ 
Thanks to the assumption \ref{assumptionA2}, we can take $\Omega_2$ so small that $\overline{\Omega_2} \setminus \Omega_1$ is a closed subset of the open set $\Omega_a^-$. 
In particular, we have $a^- \geq c$ in $\overline{\Omega_2} \setminus \Omega_1$ for some $c>0$. 
Since $\overline{\Omega_2} \setminus \Omega_1$ is compact, we use Lemma~\ref{dc} together with a covering argument and the uniform $L^\infty(\Omega)$-boundedness of solutions of \eqref{eq:Pmu} given by Proposition~\ref{pg2}~\ref{pg23} to derive the existence of $\mu^*>0$ such that any nonnegative solution $u$ of \eqref{eq:Pmu} vanishes in $\overline{\Omega_2} \setminus \Omega_1$ whenever $\mu > \mu^*$. 
On the other hand, we have $a \leq 0$ in $\Omega \setminus \Omega_2$ by our assumptions. 
Recalling that $u \in C^1(\Omega)$ and applying the strong maximum principle \cite{V}, it is not hard to deduce that $u$ must also vanish in $\Omega \setminus \Omega_2$. 
\end{proof}

Now we are ready to prove Theorem~\ref{thm:main1}.
\begin{proof}[Proof of Theorem~\ref{thm:main1}] 
Given $\epsilon >0$, we denote by $\omega_{i,\epsilon}$ the $\epsilon$-neighbourhood of $\omega_i$, i.e., $\omega_{i,\epsilon}:=\{x \in \Omega:\, \mathrm{dist}(x,\omega_i)<\epsilon\}$. 
From the assumption \ref{assumptionA2} there exists the largest $\epsilon^*>0$ such that for any $\epsilon_0 \in (0,\epsilon^*)$ and $i \neq j$ we have 
$\overline{\omega_{i,\epsilon_0}} \cap \overline{\omega_{j,\epsilon_0}} = \emptyset$. 
Take any such $\epsilon_0$ and set $\Omega_i:=\omega_{i,\epsilon_0}$ for $i=1,\dots,n$.
Proposition~\ref{cor:wcp} gives the existence of $\mu^*(\epsilon_0)$ such that the nonnegative ground state solution $U_\mu$ of \eqref{eq:Pmu} satisfies $U_\mu \equiv 0$ in $\Omega \setminus \bigcup_i {\Omega_i}$ for any $\mu > \mu^*(\epsilon_0)$. 
We claim that $U_{\mu,i} := U_\mu \chi_{\Omega_i}$ solves \eqref{eq:Pmu} for any such $\mu$. 
Fix some $i=1,\dots,n$ and consider a cut-off function $\psi \in C_0^1(\Omega)$ such that $0\leq \psi \leq 1$ in $\Omega$, $\psi \equiv 1$ in $\Omega_i$, and $\psi \equiv 0$ in $\bigcup_{j \neq i} \Omega_j$. 
For any $\phi \in \W$ we take $\psi \phi$ as a test function in \eqref{eq:Pmu} and obtain
$$
\intO |\nabla U_\mu|^{p-2} \nabla U_\mu \nabla (\psi \phi) \,dx 
= 
\intO a_\mu U_\mu^{q-1}  \psi \phi \,dx. 
$$
Note that $U_\mu^{q-1}  \psi \phi=U_{\mu,i}^{q-1} \phi$, $\nabla U_\mu \nabla \psi \equiv 0$, and $|\nabla U_\mu|^{p-2} \nabla U_\mu \nabla \phi \, \psi=|\nabla U_{\mu,i}|^{p-2} \nabla U_{\mu,i} \nabla \phi$, which yields
$$
\intO |\nabla U_{\mu,i}|^{p-2} \nabla U_{\mu,i} \nabla \phi \,dx
= 
\intO a_\mu U_{\mu,i}^{q-1} \phi \,dx.
$$
Thus, the claim is proved, and by Proposition~\ref{p2} we infer that $U_{\mu,i}$ minimizes $I_\mu$ over $K(\bigcup_{j \neq i} \omega_j)$ and, consequently, also over $K(\Omega \setminus \Omega_i)$. 
In addition, since $U_{\mu,i} U_{\mu,j} \equiv 0$ for $i \neq j$, we see that $\sum_{i \in J} U_{\mu,i}$ also solves \eqref{eq:Pmu} for any $\mu > \mu^*(\epsilon_0)$ and $J \subset \{1,\dots,n\}$.
Therefore,  \eqref{eq:Pmu} has at least $2^n-1$ nontrivial nonnegative solutions for $\mu > \mu^*(\epsilon_0)$. 
From Proposition~\ref{pg2}~\ref{pg21} we deduce that \eqref{eq:Pmu} has exactly $2^n-1$ nontrivial nonnegative solutions for such $\mu$. 
The pointwise monotonicity of $U_\mu$ given in Proposition~\ref{prop:prop}~\ref{prop:prop:1} shows that the mapping $\mu \mapsto \mathrm{supp}\, U_{\mu,i}$ is nonincreasing for $i=1,\dots,n$. 
Therefore, we have $\mu^*(\epsilon_0) \geq  \mu^*(\epsilon_1)$ for $\epsilon_1 \in (\epsilon_0,\epsilon^*)$. 
Defining $\mu_\infty = \lim_{\epsilon \nearrow \epsilon^*} \mu^*(\epsilon)$, we conclude that \eqref{eq:Pmu} has exactly $2^n-1$ nontrivial nonnegative solutions for $\mu > \mu_\infty$. 
In view of Theorem~\ref{thm:main0}, we also get  $\mu_\infty \geq \mu_0$.

Let us justify the Hausdorff convergence of the support of $U_{\mu,i}$.
The Hausdorff distance between two compact subsets $K_1$ and $K_2$ of $\mathbb{R}^N$ is defined as
$$
d^H(K_1,K_2)
=
\max\{\rho(K_1,K_2),\rho(K_2,K_1)\},
$$
where $\rho(K_1,K_2) = \sup_{x \in K_1} \text{dist}(x,K_2)$, see, e.g., \cite[Section~2.2.3]{hp}. 
Clearly, if $K_1 \subset K_2$, then $\rho(K_1,K_2) = 0$, and hence $d^H(K_1,K_2) = \rho(K_2,K_1)$. 
It is well-known that the Hausdorff distance is a metric on the set of all non-empty compact subsets of $\mathbb{R}^N$. 
Moreover, the set of all non-empty compact subsets of the bounded domain $\Omega$ is compact with respect to $d^H$, see, e.g., \cite[Theorem~2.2.25]{hp}. 
Consequently, one can speak about the Hausdorff convergence of compact subsets of $\Omega$. 

Since the support of each $U_{\mu,i}$ is a compact subset of $\overline{\Omega_i}$ for $\mu>\mu_\infty$, we have
$$
d^H(\mathrm{supp}\, U_{\mu,i},\overline{\omega_i})
=
\rho(\overline{\omega_i},\mathrm{supp}\, U_{\mu,i})
\leq 
\rho(\overline{\omega_i},\overline{\Omega_i})
=
d^H(\overline{\Omega_i},\overline{\omega_i}).
$$
Recalling that $\Omega_i$ is the $\epsilon_0$-neighbourhood of $\omega_i$, it is not hard to see that 
$d^H(\overline{\Omega_i},\overline{\omega_i}) \to 0$
as $\epsilon_0 \to 0$,
see, e.g., \cite[Section~2.2.3.2 (5)]{hp}. Therefore, we conclude that $\mathrm{supp}\, U_{\mu,i}$ converges, in the sense of Hausdorff, to $\overline{\omega_{i}}$ as $\mu \to +\infty$. 

Finally, the convergence of $U_\mu$ towards $U_\infty$ given in Proposition~\ref{prop:prop} \ref{prop:prop:1} and the discussion in  Remark~\ref{rem:Uinfty} imply that $U_{\mu,i} \to U_{\infty}\chi_{\Omega_i} \equiv U_{\infty}\chi_{\omega_i}$ in 
$\W$ as $\mu \to +\infty$. 
\end{proof}

\begin{remark}\label{rem:reg} 
It can be shown that each bump $U_{\mu,i}$ is a solution of \eqref{eq:Pmu} in a different way than above. 
Namely, recall that the nonnegative ground state solution $U_\mu$ of \eqref{eq:Pmu} satisfies \eqref{eq:W12loc}. 
If $\Omega_{i}$ is sufficiently smooth, then, taking any $\xi \in C_0^\infty(\Omega)$ and integrating by parts (see, e.g., \cite[Chapter~I, Theorem~1.2]{temam}), we get
\begin{equation}\label{eq:weakstrong}
\int_{\Omega_{i}} |\nabla U_\mu|^{p-2} \nabla U_\mu \nabla \xi \,dx
+
\int_{\Omega_{i}} \text{div}(|\nabla U_\mu|^{p-2} \nabla U_\mu)  \, \xi \,dx
=
\int_{\partial \Omega_{i}} |\nabla U_\mu|^{p-2} \, \frac{\partial U_\mu}{\partial \nu} \, \xi \, dS.
\end{equation}
Since $U_\mu \in C^1(\Omega)$ and $U_\mu$ vanishes on $\partial \Omega_{i}$, the integral on the right-hand side of \eqref{eq:weakstrong} is zero.  
Therefore, 
\begin{align}
\int_{\Omega} |\nabla U_{\mu,i}|^{p-2} \nabla U_{\mu,i} \nabla \xi \,dx
&=
\int_{\Omega_{i}} |\nabla U_{\mu}|^{p-2} \nabla U_{\mu} \nabla \xi \,dx
\\
&=
-\int_{\Omega_{i}} \text{div}(|\nabla U_{\mu}|^{p-2}\nabla U_{\mu}) \xi \,dx
=
\int_{\Omega_{i}} a_\mu U_{\mu}^{q-1} \xi \,dx
=
\int_{\Omega} a_\mu U_{\mu,i}^{q-1} \xi \,dx.
\end{align}
That is, $U_{\mu,i}$ is a nonnegative (weak) solution of \eqref{eq:Pmu}. 
Note that the smoothness of $\Omega_{i}$ is a non-restrictive assumption. 
Indeed, since $\overline{\Omega_{i}}$ (or $\overline{\omega_{i}}$) can be seen as the zero set of a nonnegative function $f \in C^\infty(\mathbb{R}^N)$, it can be approximated by sublevel sets of $f$, and almost all of these sublevel sets are smooth domains due to Sard's theorem.
\end{remark}

\begin{remark}
One may naturally ask what kind of critical points of $I_\mu$ are the bumps $U_{\mu,i}$. 
It is easy to see that they are not local minimizers of $I_\mu$, since any small perturbation along any other bump direction decreases $I_\mu$. On the other hand, Proposition \ref{p2} shows that $U_{\mu,i}$ minimizes $I_\mu$ over $K(\bigcup_{j\neq i}\omega_j)$. 
In particular, each $U_{\mu,i}$ is a saddle point of $I_\mu$.
\end{remark}

\section{Further high multiplicity results}\label{sec:extensions}

In this section, we deal with the problem
\begin{equation}\label{eq:Qmu}
	\tag{$\mathcal{Q}_\mu$}
	\left\{
	\begin{aligned}
		-\Delta_p u &= a_\mu(x)(|u|^{q-2}u+|u|^{r-2}u) 
		&&\text{in}\ \Omega, \\
		u &=0 &&\text{on}\ \partial \Omega,
	\end{aligned}
	\right.
\end{equation}
where $1<q<p\leq r$, $\mu>0$, and the sign-changing weight $a$ satisfies the same assumptions as in Section~\ref{sec:intro}. 

\subsection{The case \texorpdfstring{$r=p$}{r=p}}

For $\mu>0$ we consider 
$$
\lambda_1(\mu):=\inf\left\{ \intO \left(|\nabla u|^p - a_\mu |u|^p\right) dx:~ u \in \W,~ \|u\|_p=1\right\}
$$
and its formal limit as $\mu \to +\infty$:
$$
\lambda_\infty:=\inf\left\{ \intO \left(|\nabla u|^p - a^+ |u|^p\right) dx:~ u \in K(\Omega_a^-),~ \|u\|_p=1 \right\},
$$
where the subset $K(\Omega_a^-)$ of $\W$ is defined in \eqref{eq:K}.  
It is not hard to see that both $\lambda_1(\mu)$ and  $\lambda_\infty$ are attained. 
Moreover, as is well known, $\lambda_1(\mu)$ is the first eigenvalue of the problem 
$$
-\Delta_p u - a_\mu |u|^{p-2}u = \lambda |u|^{p-2}u, \quad u \in \W,
$$ 
and it is achieved by a unique $\phi_1(\mu) \in \W$ such that $\phi_1(\mu)>0$ in $\Omega$ and $\|\phi_1(\mu)\|_p=1$.

\begin{lemma}\label{p3}
	The map $\mu \mapsto \lambda_1(\mu)$ is increasing and $\lambda_\infty=\displaystyle \lim_{\mu \to +\infty} \lambda_1(\mu)$.    
\end{lemma}
\begin{proof}
	For any $\mu>\mu'>0$ we have
	$$
	\lambda_1(\mu)
	=
	\intO \left(|\nabla \phi_1(\mu)|^p - a_\mu \phi_1(\mu)^p\right) dx
	>
	\intO \left(|\nabla \phi_1(\mu)|^p - a_{\mu'} \phi_1(\mu)^p\right) dx
	\geq 
	\lambda_1(\mu'),
	$$ 
	so that $\mu \mapsto \lambda_1(\mu)$ is increasing. 
	The inequality $\lambda_\infty > \lambda_1(\mu)$ holds for any $\mu>0$ in a similar way.
	Indeed, denoting by $\phi_\infty \in K(\Omega_a^-)$ a minimizer of $\lambda_\infty$ and recalling that $\phi_1(\mu)>0$ in $\Omega$, we get
	\begin{equation}\label{eq:lll}
		\lambda_\infty
		=
		\intO \left(|\nabla \phi_\infty|^p - a^+ \phi_\infty^p\right) dx
		=
		\intO \left(|\nabla \phi_\infty|^p - a_\mu \phi_\infty^p\right) dx
		>
		\lambda_1(\mu).
	\end{equation}
	Let $\mu_k \to +\infty$ and $u_k:=\phi_1(\mu_k)$. 
	It is not hard to deduce from \eqref{eq:lll} that $\{u_k\}$ is bounded in $\W$, so we may assume that, along a subsequence, $u_k \to u$ weakly in $\W$ and strongly in $L^p(\Omega)$, for some nonnegative $u \in \W$. 
	Moreover, since
	$$
	\intO \left(|\nabla u_k|^p - a^+ u_k^p + \mu_k a^- u_k^p\right) dx=\lambda_1(\mu_k)\leq \lambda_\infty,
	$$ 
	we obtain
	$$
	\intO a^- u^p \,dx= \lim_{k \to +\infty} \intO a^- u_k^p \,dx =0,
	$$
	which yields $u \in K(\Omega_a^-)$. 
	Thus, since $\|u\|_p=1$, we conclude that
	\begin{align}
		\lambda_\infty \leq \intO \left(|\nabla u|^p - a^+ u^p\right) dx &\leq \liminf_{k \to +\infty} \intO \left(|\nabla u_k|^p - a^+ u_k^p\right) dx\\ &\leq \liminf_{k \to +\infty} \intO \left(|\nabla u_k|^p - a_\mu u_k^p\right) dx
		\leq \limsup_{k\to +\infty} \lambda_1(\mu_k) \leq \lambda_\infty.
	\end{align}
That is, $\lambda_1(\mu) \to \lambda_\infty$ as $\mu \to +\infty$. 
\end{proof}

\begin{corollary}\label{cc}
	Assume that $\lambda_\infty>0$. 
	Then there exists $\mu_{\infty}>0$
	and $C>0$ 
	such that  
	\begin{equation}\label{eq:p3:1}
		\intO \left(|\nabla u|^p - a_\mu |u|^p\right) dx \geq C \, \|\nabla u\|_p^p
	\end{equation}
	for any $u \in \W$ and $\mu \geq \mu_{\infty}$.
\end{corollary}
\begin{proof}
	The assumption $\lambda_\infty>0$ and Lemma~\ref{p3} imply the existence of  $\mu_{\infty}>0$ such that $\lambda_1(\mu_{\infty})>0$.
	Since the map $\mu \mapsto -a_\mu(x)$ is nondecreasing for any $x \in \Omega$, it suffices to show that there exists $C>0$ such that  $$
	\intO \left(|\nabla u|^p - a_{\mu_{\infty}} |u|^p\right) dx \geq C\|\nabla u\|_p^p
	$$ 
	for any $u \in \W$. 
	Suppose, by contradiction, that we can find a sequence $\{u_k\} \subset \W$ satisfying $\|\nabla u_k\|_p=1$ and 
	\begin{equation}\label{eq:cc:1}
	\intO \left(|\nabla u_k|^p - a_{\mu_{\infty}} |u_k|^p\right) dx \to 0.
	\end{equation}
	Thus, we may assume that, along a subsequence,  $u_k \to u$ weakly in $\W$ and strongly in $L^p(\Omega)$, for some $u \in \W$. 
	Therefore, we have
	\begin{equation}\label{eq:cor1}
		\intO \left(|\nabla u|^p - a_{\mu_{\infty}} |u|^p\right) dx \leq 0.
	\end{equation}
	Recalling that each $\|\nabla u_k\|_p=1$, it is not hard to conclude from \eqref{eq:cc:1} and the strong convergence of $\{u_n\}$ in $L^q(\Omega)$ that $u \not \equiv 0$. 
	Consequently, \eqref{eq:cor1} yields $\lambda_1(\mu_{\infty}) \leq 0$, and we reach a contradiction.    
\end{proof}

Consider the functional associated with the problem \eqref{eq:Qmu}:
$$
I_{\mu,p}(u)=\frac{1}{p}\intO \left(|\nabla u|^p-a_{\mu}|u|^p\right) dx
-
\frac{1}{q}\intO a_{\mu}|u|^q \,dx, \quad u \in \W.
$$

\begin{proposition}\label{p4}
	Let \ref{assumptionA1} be satisfied and $\lambda_\infty>0$.
	Then there exists $\mu_{\infty}>0$ such that for $\mu\geq \mu_{\infty}$ the functional $I_{\mu,p}$ has a nontrivial nonnegative minimizer $V_\mu$. Moreover, the following assertions hold:
	\begin{enumerate}[label={\rm(\arabic*)}]
		\item $V_\mu>0$ in $\omega_i$ for any $i=1,\dots,n$.
		\item\label{p4:2} The family $\{V_\mu:\, \mu \geq \mu_{\infty}\}$ is bounded in $\W \cap L^{\infty}(\Omega)$.
	\end{enumerate}
\end{proposition}
\begin{proof}
	Let $\mu_{\infty}$ be given by Corollary~\ref{cc}. Then for any $\mu \geq \mu_{\infty}$ the functional $I_{\mu,p}$ is coercive and thus has a global minimizer $V_\mu \geq 0$, which satisfies $I_{\mu,p}(V_\mu)<0$, so that $V_\mu$ is nontrivial.  
	Arguing as in Proposition~\ref{prop:prop}, we see that $V_\mu>0$ in $\omega_i$ for any $i=1,\dots,n$. Finally, from Corollary~\ref{cc}, the equality $I_{\mu,p}'(V_\mu)V_\mu=0$, and the Rellich-Kondrachov theorem we get
	$$
	C\|\nabla V_\mu\|_p^p \leq \intO \left(|\nabla V_\mu|^p - a_\mu |V_\mu|^p\right) dx
	=
	\intO a_\mu V_\mu^q \,dx
	\leq 
	\intO a^+ V_\mu^q \,dx \leq C'\|\nabla V_\mu\|_p^q,
	$$ 
	where $C, C'>0$ do not depend on $\mu$. 
	This implies that $\{V_\mu:\, \mu \geq \mu_{\infty}\}$ is bounded in $\W$. The standard bootstrap argument yields the boundedness in  $L^\infty(\Omega)$. 
\end{proof}

\begin{remark}
	In contrast to the case of the problem~\eqref{eq:Pmu}, we do not know if $V_\mu$ is unique (whenever it exists), i.e., if $I_{\mu,p}$ has a unique nonnegative minimizer, cf.\ Proposition~\ref{p2}.  
	The same issue arises for the limit functional 
	$$
	I_{0,p}(u)=\frac{1}{p}\intO \left(|\nabla u|^p-a^+|u|^p\right) dx-\frac{1}{q}\intO a^+|u|^q \,dx, \quad u \in \W.
	$$
	However, one can show that if $\mu_k \to +\infty$, then the sequence of any nonnegative ground state solutions $\{V_{\mu_k}\}$ converges in $\W$ to a nonnegative minimizer of $I_{0,p}$ over $K(\Omega_a^-)$.
\end{remark}

We have the following counterpart of Proposition~\ref{cor:wcp}. 
\begin{proposition}\label{dccc1}
	Let \ref{assumptionA1} and \ref{assumptionA2} be satisfied, and $r=p$. 
	Let $\widetilde{\Omega} \subset \Omega$ be a neighbourhood of $\bigcup_{i} \omega_i$. 
	Then there exists $\mu^*>0$ such that any nonnegative solution of \eqref{eq:Qmu} vanishes in $\Omega \setminus \widetilde{\Omega}$ for $\mu>\mu^*$.
\end{proposition}
\begin{proof}
	Let  $B=B(x_0,R) \subset \Omega_a^-$. Since nonnegative solutions of \eqref{eq:Qmu} satisfy $-\Delta_p u \leq -A u^{q-1}$ in $B$ for some $A>0$ and they are uniformly bounded in $L^\infty(\Omega)$ by Proposition~\ref{p4} \ref{p4:2}, Lemma~\ref{dc} applies to $V_\mu$, which yields the desired conclusion exactly in the same way as in the proof of Proposition~\ref{cor:wcp}.
\end{proof}

Repeating the argument from the proof of Theorem~\ref{thm:main1}, we deduce the following result:

\begin{theorem}\label{thm:main2}
	Let \ref{assumptionA1} and  \ref{assumptionA2} be satisfied. Assume that $r=p$ and $\lambda_\infty>0$.
	Then there exists $\mu_{\infty}>0$ such that \eqref{eq:Qmu} has at least $2^n-1$ nontrivial nonnegative solutions (all of which are dead core solutions) for $\mu>\mu_{\infty}$. 
	Moreover, as $\mu \to +\infty$, the support of each such solution converges, in the sense of Hausdorff, to  $\bigcup_{i\in \mathcal{J}} \overline{\omega_i}$ for some $\mathcal{J} \subset \{1,\dots,n\}$. 
\end{theorem}

Note that the condition $\lambda_\infty>0$ is satisfied if $\|a^+\|_\infty>0$ is small enough, so that Theorem \ref{t2} in the case $r=p$ follows from Theorem \ref{thm:main2}.

The presented multiplicity result can be extended to a larger class of problems as follows: let $b \in L^{\infty}(\Omega)$ and set 
$$
\lambda_\infty(b):=\inf\left\{ \intO \left(|\nabla u|^p - b(x) |u|^p\right) dx:~ u \in K(\Omega_a^-),~ \|u\|_p=1 \right\}.
$$

\begin{theorem}\label{thm:main2'}
	Let \ref{assumptionA1} and  \ref{assumptionA2} be satisfied. Assume that $b\leq 0$ in $\Omega_a^-$ and $\lambda_\infty(b)>0$.
	Then there exists $\mu_{\infty}>0$ such that the problem $$-\Delta_p u =a_\mu(x)|u|^{q-2}u + b(x)|u|^{p-2}u, \quad u \in \W,$$ has at least $2^n-1$ nontrivial nonnegative solutions (all of which are dead core solutions) for  $\mu>\mu_{\infty}$. 
	Moreover, as $\mu \to +\infty$, the support of each such solution converges, in the sense of Hausdorff, to  $\bigcup_{i\in \mathcal{J}} \overline{\omega_i}$ for some $\mathcal{J} \subset \{1,\dots,n\}$. 
\end{theorem}

\subsection{The case \texorpdfstring{$r>p$}{r>p}}

Assume first that $r<p^*$.
We show that if $\|a^+\|_\infty>0$ is small enough, then for any $\mu>0$ the problem \eqref{eq:Qmu} has a nontrivial nonnegative ground state solution $V_\mu$, which is characterized by
$I_{\mu,r}(V_\mu)=\min_{\mathcal{N}_\mu} I_{\mu,r}$, where 
$$
I_{\mu,r}(u)
=
\frac{1}{p}\intO |\nabla u|^p \,dx-\frac{1}{q}\intO a_{\mu}|u|^q \,dx-\frac{1}{r}\intO a_{\mu}|u|^r \,dx, \quad u \in \W,
$$
and 
$$
\mathcal{N}_\mu
=
\{u\in \W\setminus \{0\}:~ I_{\mu,r}'(u)u=0\}
$$ 
is the Nehari manifold corresponding to $I_{\mu,r}$. 

\begin{proposition}\label{p5}
	Let \ref{assumptionA1} be satisfied and $r<p^*$. 
	Then there exists $\delta>0$ such that if $\|a^+\|_\infty \in (0,\delta)$, then \eqref{eq:Qmu}
	has a nontrivial nonnegative ground state solution $V_\mu$ for  $\mu>0$. Moreover, the following assertions hold:
	\begin{enumerate}[label={\rm(\arabic*)}]
		\item\label{p5:1} $V_\mu>0$ in $\omega_i$ for any $i=1,\dots,n$.
		\item\label{p5:2} The family $\{V_\mu:\, \mu >0\}$ is bounded in $\W \cap L^{\infty}(\Omega)$.
	\end{enumerate}
\end{proposition}
\begin{proof}
    First, by the Rellich-Kondrachov theorem one can find constants $C_1, C_2>0$ such that 
    \begin{align}
    I_{\mu,r}(u)
    \geq
    \frac{1}{p} \|\nabla u\|_p^p
    -
    C_1 \|a^+\|_\infty \|\nabla u\|_p^q
    -
    C_2 \|a^+\|_\infty \|\nabla u\|_p^r
    \quad \text{for any}~ u \in \W ~\text{and}~ \mu>0.
    \end{align}
    It is not hard to see that if $\|a^+\|_\infty>0$ is small enough, then for any $u \in \W$ and $\mu>0$ there exists $t>0$ such that $I_{\mu,r}(tu)>0$. 
    Therefore, for a fixed $\mu>0$, one can deduce the following behaviour of the polynomial type function $t \mapsto I_{\mu,r}(tu)$:
    \begin{enumerate}[label={\rm(\roman*)}]
        \item\label{eq:p5:1} If $\intO a_\mu|u|^q \,dx>0$ and $\intO a_\mu|u|^r \,dx> 0$, then there exist exactly two values $0<t^+(u)<t^-(u)$ such that $t^+(u)u,t^-(u)u \in \mathcal{N_\mu}$. 
        Moreover, we have 
        $$
        I_{\mu,r}(t^+(u)u)<0<I_{\mu,r}(t^-(u)u).
        $$
        \item\label{eq:p5:2} If $\intO a_\mu|u|^q \,dx>0$ and $\intO a_\mu|u|^r \,dx\leq 0$, then there exists a unique $t^+(u)>0$ such that $t^+(u)u \in \mathcal{N_\mu}$. Moreover, $I_{\mu,r}(t^+(u)u)<0$.
        \item If $\intO a_\mu|u|^q \,dx\leq 0$, then for any $t>0$ such that $tu \in \mathcal{N_\mu}$ we have $I_{\mu,r}(tu)>0$.
    \end{enumerate} 
    Thus, we conclude that if $\|a^+\|_\infty>0$ is small enough, then, for a fixed $\mu>0$, 
    $$
    \inf_{\mathcal{N_\mu}} I_{\mu,r}
    =
    \inf
    \Big\{
    I_{\mu,r}(t^+(u)u):~ u \in \W,~ \intO a_\mu|u|^q \,dx>0
    \Big\} 
	< 0, 
    $$ 
    and a standard argument shows that this infimum is attained by some nontrivial $V_\mu \geq 0$. 

\ref{p5:1} By the strong maximum principle, if $V_\mu$ vanishes somewhere inside $\Omega_a^+$, then it does so in a connected component of this set. 
In that case, we take a nontrivial nonnegative $\phi \in C_0^1(\Omega)$ supported in a ball $B \subset \Omega_a^+$ such that $V_\mu =0$ in $B$. 
Since $\intO a_\mu \phi^q \,dx>0$, we get from \ref{eq:p5:1} or \ref{eq:p5:2} the existence of $t^+(\phi) > 0$ such that 
\begin{equation}\label{eq:p5:xx}
t^+(\phi)\phi \in \mathcal{N}_\mu
\quad \text{and} \quad
I_{\mu,r}(t^+(\phi)\phi) < 0.
\end{equation}
Noting that $V_\mu$ and $\phi$ have disjoint supports and $V_\mu \in \mathcal{N}_\mu$, we get
$$
I_{\mu,r}'(V_\mu+t^+(\phi)\phi)(V_\mu+t^+(\phi)\phi) 
= 
I_{\mu,r}'(V_\mu)(V_\mu) 
+ 
I_{\mu,r}'(t^+(\phi)\phi)(t^+(\phi)\phi) = 0,
$$
that is, $V_\mu+t^+(\phi)\phi \in \mathcal{N}_\mu$.
However, using again the fact that the supports of $V_\mu$ and $\phi$ are disjoint, we derive from \eqref{eq:p5:xx} the following contradiction: 
$$
I_{\mu,r}(V_\mu) = \inf_{\mathcal{N_\mu}} I_{\mu,r}
\leq
I_{\mu,r}(V_\mu+t^+(\phi)\phi) 
=
I_{\mu,r}(V_\mu)
+ 
I_{\mu,r}(t^+(\phi)\phi)
< I_{\mu,r}(V_\mu).
$$
Therefore, we have $V_\mu>0$ in $\Omega_a^+$. 
Finally, the positivity of $V_\mu$ in every $\omega_i$ follows from the strong maximum principle.

\ref{p5:2} 
Since $I_{\mu,r}'(V_\mu)V_\mu = 0$, we have
$$
I_{\mu,r}(V_\mu)
=
\frac{r-p}{pr} \|\nabla V_\mu\|^p 
- 
\frac{r-q}{rq} \intO a_\mu V_\mu^q \,dx< 0.
$$
By the Rellich-Kondrachov theorem we get
$$
\|\nabla V_\mu\|^p 
\leq 
\frac{(r-q)p}{(r-p)q} \,
\|a^+\|_\infty \,
C \, \|\nabla V_\mu\|^q,  
$$
where $C>0$ does not depend on $\mu$. 
which shows that the family $\{V_\mu:\, \mu >0\}$ is bounded in $\W$. 
Applying the bootstrap procedure, we conclude the boundedness in $L^\infty(\Omega)$. 
\end{proof}

We note that the dead core property stated in Proposition~\ref{dccc1} remains valid in the case $p < r < p^*$, thanks to Proposition~\ref{p5} \ref{p5:2}. 
The same arguments as in the proof of Theorem~\ref{thm:main1} can be applied to $V_\mu$, providing us with the proof of Theorem \ref{t2} in the case $p<r<p^*$. 
The supercritical case $r\geq p^*$ (when $p<N$) is covered by the following remark.

\begin{remark}\label{rf}
    Theorem~\ref{t2} for any $r \geq p$ can also be proved by dealing with a nontrivial nonnegative solution $W_\mu$ of \eqref{eq:Qmu} obtained via the sub- and supersolutions method.  
    Indeed, on the one hand, the function $c \sum_i \phi_{1,B_i}$ serves as a subsolution of \eqref{eq:Qmu}, where $c>0$ is sufficiently small and each $\phi_{1,B_i}$ is a positive first eigenfunction of $-\Delta_p$ in $W_0^{1,p}(B_i)$ with $B_i$ being a ball such that $\overline{B_i}\subset \omega_i \cap \Omega_a^+$, see the proof of Proposition~\ref{prop:prop} \ref{prop:prop:1}. 
    On the other hand, a supersolution of \eqref{eq:Qmu} is given (under a smallness condition on $\|a^+\|_\infty$) by $Me$, where $M>0$ and $e \in \W$ satisfies $-\Delta_p e=1$, cf.\ the proof of \cite[Theorem~2.1]{DGU3}. Then the sub- and supersolutions method shows that, whenever $\|a^+\|_\infty>0$ is small enough, \eqref{eq:Qmu} has for any $\mu>0$ a solution $W_\mu$ such that 
    $$
    \textstyle 0\leq c \sum_i \phi_{1,B_i} \leq W_\mu \leq Me \quad \text{in}~ \Omega.
    $$ 
	Thus, by the strong maximum principle, $W_\mu>0$ in $\omega_i$ for every $i=1,\dots,n$. Lastly, since $\|W_\mu\|_\infty \leq M\|e\|_\infty$, the family $\{W_\mu:\, \mu >0\}$ is bounded in $L^{\infty}(\Omega)$, and consequently Proposition~\ref{dccc1} holds for $W_\mu$ as well. 
    In comparison with the one based on the ground state solution $V_\mu$, this procedure has the advantage of being valid even if $r\geq p^*$. 
\end{remark}

To conclude, let us observe that Theorem~\ref{t2} has the following extension:
\begin{theorem}
	Let \ref{assumptionA1} and  \ref{assumptionA2} be satisfied, and let $b \in L^{\infty}(\Omega)$ be such that $b\leq 0$ in $\Omega_a^-$.
    Then there exists  $\delta>0$ such that whenever $\|a^+\|_\infty \in (0,\delta)$ the following result holds: there exists $\mu_{\infty}>0$ such that the problem
    $$
    -\Delta_p u =a_\mu(x)|u|^{q-2}u + b(x)|u|^{r-2}u, \quad u \in \W,
    $$ 
    has at least $2^n-1$ nontrivial nonnegative solutions (all of which are dead core solutions) for  $\mu>\mu_{\infty}$. 
    Moreover, as $\mu \to +\infty$, the support of each such solution converges, in the sense of Hausdorff, to  $\bigcup_{i\in \mathcal{J}} \overline{\omega_i}$ for some $\mathcal{J} \subset \{1,\dots,n\}$.
\end{theorem}

\end{document}